\documentclass{amsart}

\usepackage{amssymb,color}
\usepackage{amsfonts}
\usepackage{amsmath}
\usepackage{euscript}
\usepackage{enumerate}
\usepackage{pdfsync}
\newtheorem{theorem}{Theorem}[section]
\newtheorem{lemma}[theorem]{Lemma}
\newtheorem{note}[theorem]{Note}

\newtheorem{cor}[theorem]{Corollary}

\newtheorem*{Theorem1'}{Theorem 1'}

\theoremstyle{definition}

\theoremstyle{remark}

\newcommand \Z{{\mathbb Z}}
\newcommand \Q{{\mathbb Q}}
\newcommand \N{{\mathbb N}}
\newcommand \GL{{\mathrm{GL}}}

\begin{document}

\title[Linear diophantine equations in several variables]{Linear diophantine equations in several variables}

\author{R. Quinlan}
\address{School of Mathematical and Statistical Sciences, National University of Ireland, Galway, Ireland}
\email{rachel.quinlan@nuigalway.ie}

\author{M. Shau}
\address{Department of Mathematics and Statistics, University of Regina, Canada}
\email{tuktukishau@gmail.com}

\author{F. Szechtman}
\address{Department of Mathematics and Statistics, University of Regina, Canada}
\email{fernando.szechtman@gmail.com}
\thanks{The third author was supported in part by an NSERC discovery grant}





\begin{abstract} Let $R$ be a ring and let $(a_1,\dots,a_n)\in R^n$ be a unimodular vector,  where 
$n\geq 2$ and each $a_i$ is in the center of $R$. Consider the
linear equation $a_1X_1+\cdots+a_nX_n=0$, with solution set $S$. Then $S=S_1+\cdots+S_n$, where each $S_i$ is naturally
derived from $(a_1,\dots,a_n)$, and we give a presentation of $S$ in terms of generators taken
from the~$S_i$ and appropriate relations. 
Moreover, under suitable assumptions, we elucidate the structure of each quotient module $S/S_i$. 
Furthermore, assuming that $R$ is a principal ideal domain, we provide a simple way to construct a basis of $S$
and, as an application, we determine the structure of the quotient module $S/U_i$, where each $U_i$
is a specific module containing $S_i$.
\end{abstract}

\maketitle

\section{Introduction}

Let $R$ be a principal ideal domain and let $(a_1,\dots,a_n)\in R^n$ be a unimodular vector, that is, 
$R=Ra_1+\cdots+Ra_n$, where $n\geq 2$. Let $S$ be the submodule of $R^n$ of all solutions $(x_1,\dots,x_n)\in R^n$
to the homogeneous linear equation 
\begin{equation}\label{ecu}
a_1X_1+\cdots+a_nX_n=0.
\end{equation}
In this paper, we study structural properties of several quotients of $S$ naturally arising from (\ref{ecu}),
we describe $S$ by means of generators and relations, and we construct a basis of $S$ directly from the 
coefficients in (\ref{ecu}).

Our interest is primarily in the extent to which $S$ may be described directly and irredundantly in terms of $a_1,\dots ,a_n$. Abstractly, $S$ is a free module of rank $n-1$ over $R$, and so it admits a basis consisting of $n-1$ elements. We begin with a collection of $\binom{n}{2}$ elements that span $S$ and are directly defined in terms of the $a_i$. This collection does not in general contain a basis of $S$ as a subset, but it contains bases for submodules of $S$ of full rank corresponding to each nonzero~$a_i$. A small adjustment to the original spanning set motivates the definition of another family of submodules, and we study the quotient structures in all cases.

Much of the literature on (systems of) linear equations over rings is concerned with the problem of efficiently computing and describing the non-negative solutions of systems of equations written over $\Z$, a problem of interest in linear programming and in combinatorial optimization. From the linear programming point of view, numerous authors have proposed algorithms for determining tractable generating sets for the monoid of non-negative solutions; see for example \cite{CF}, \cite{PV}, \cite{CD} and the references therein, for some different approaches to this problem. From a commutative algebra viewpoint, Stanley \cite{S} has given a detailed theoretical analysis of the monoid of non-negative solutions to a system of linear equations over $\Z$. In the case of a single homogeneous equation, the influence of the coefficients of (\ref{ecu}) on the factorial properties of the monoid of non-negative solutions is investigated by Chapman, Krause and Oeljeklaus in \cite{CKO}.

In our context, the equation (\ref{ecu}) is written over a ring $R$ that may not have an order relation, and our object of study is the $R$-module $S$ of all solutions. While the list of coefficients of (\ref{ecu}) does not readily translate to a basis of $S$, it does provide very simple bases for several submodules of $S$ whose quotients are amenable to analysis, yielding a method to construct a basis of $S$ from the $a_i$. Our starting point is a set of $\binom{n}{2}$ vectors in $R^n$, defined as follows. We remark that the same collection of vectors is considered by Kryvyi \cite{K}, as an ingredient in the study of systems of linear equations over $\Z$.
We are confused by the use of the term \emph{basis} in \cite[Theorem 2]{K}, as a basis of $S$ cannot contain $n$ vectors.

For $1\leq i\neq j\leq n$, let $v(i,j)\in S$ be the vector whose $i$th and $j$th
entries are respectively equal to $-a_j$ and $a_i$, and whose other entries are all equal to~0, observing that
$v(j,i)=-v(i,j)$. For a fixed $1\leq i\leq n$, let $S_i$ be the span of all $v(i,j)$ 
with $j\neq i$. We further set  $N=\{1,\dots,n\}$ and for any subset $M$ of $N$, we define 
$$
I_M=\underset{i\in M}\sum Ra_i,\; S_M=\underset{i\in M}\sum S_i,\; B_M=\{v(i,j)\,|\, 1\leq i<j\leq n, \{i,j\}\cap M\neq\emptyset\}.
$$
Depending on the choice of $(a_1,\dots,a_n)$, there may be several subsets $M$ of $N$ such that $I_M=R$.
We fix one of them and set $m=|M|$.

The fact that $I_M=R$ readily implies $S_M=S$ (see Lemma \ref{gen}), whence $S$ is spanned by $B_M$. Theorem \ref{rela}
gives a presentation of $S$ in terms of the generators from $B_M$. An alternative proof is given in Theorem \ref{rela2}.
The structure of each quotient module $S/S_i$ is elucidated in Theorem \ref{S/S_1}.

Now $S$ is a free $R$-module of rank $n-1$ and we consider the problem of constructing a basis of $S$ in a straightforward manner.
It is clear that if $m=1$ then $B_M$ is already a basis of $S$. If $m=2$ we give a closed formula to produce a basis
of  $S$. Theorem \ref{basis} extends this formula and
provides a general and simple answer to this problem for any $m$. We demonstrate the power of Theorem~\ref{basis} by determining
the structure of each quotient module $S/U_i$ such that $a_i\neq 0$, where 
$U_i$ is the span of all vectors $v(i,j)/\gcd(a_i,a_j)$, $j\neq i$.
This turns out to be considerably more difficult than the study of $S/S_i$.

We will keep the above notation throughout the paper, although the assumption that $R$ be a principal
ideal domain is only necessary for some of our results. Thus, we will only suppose at the outset that
$R$ is an arbitrary ring with $1\neq 0$, not necessarily commutative, and that each $a_i\in Z(R)$, the center of $R$.
Further assumptions will be introduced as needed.

\section{Generators and relations for $S$}

\begin{lemma}
\label{gen}
We have $a_i S\subseteq S_i$ for all $1\leq i\leq n$, so $S=S_M$ is spanned by $B_M$.
\end{lemma}

\begin{proof} We first show that $a_1 S\subseteq S_1$. 
For this purpose, let $(x_1,x_2,\dots,x_n)\in S$. Then
$$
a_1x_1=-(a_2x_2+\cdots+a_n x_n),
$$
whence
$$
a_1(x_1,x_2,\dots,x_n)=(-(a_2x_2+\cdots+a_n x_n),a_1x_2,\dots,a_1x_n)=x_2v(1,2)+\cdots+
x_nv(1,n),
$$
as required. Likewise we show that $a_iS\subseteq S_i$, $1\leq i\leq n$,
so $a_iS\subseteq S_M$, $i\in M$.
Since $I_M=R$, it follows that $S=S_M$ and therefore $S$ is spanned by $B_M$.
\end{proof}

We know from Lemma \ref{gen} that $B_M$ generates $S$. We turn our attention to finding defining
relations among these generators.
Now, if $m=1$ then $B_M$ is a basis of~$S$, with no linear relations.
Thus, we may assume without loss that $m\geq 2$. If $n=2=m$ then from $R=Ra_1+Ra_2$ we easily see that
$B_M$ is a basis of $S$. Hence, we may also assume that $n\geq 3$. It is then clear that the set $E$ of all
triples $(i,j,k)$ such that $1\leq i<j<k\leq n$ and $|\{i,j,k\}\cap M|\geq 2$ is nonempty. For any $(i,j,k)\in E$, the vectors
$v(i,j),v(i,k),v(j,k)$ are in $B_M$ and satisfy the following relations:
\begin{equation}
\label{defrel}
a_k v(i,j)-a_j v(i,k)+a_i v(j,k)=0.
\end{equation}
We aim to show that (\ref{defrel}) are defining relations. This means the following. 

Let $D$ be the set of all pairs $(i,j)$ such that $1\leq i<j\leq n$
and $\{i,j\}\cap M\neq\emptyset$, and let $X$ be a free $R$-module 
with basis $x(i,j)$, where $(i,j)\in D$. Let $\Lambda:X\to S$ be the $R$-module epimorphism given by $x(i,j)\to v(i,j)$. 
For any triple $(i,j,k)\in E$, set 
$$
y(i,j,k)=a_k x(i,j)-a_j x(i,k)+a_i x(j,k),
$$
and let $Y$ be the $R$-span of all $y(i,j,k)$ with $(i,j,k)\in E$. It follows from (\ref{defrel}) that $Y\subseteq\ker\Lambda$. 
The assertion that (\ref{defrel}) are defining relations for $S$ means that
\begin{equation}
\label{ky}
\ker\Lambda=Y,
\end{equation}
that is,
$$
S\cong \langle x(i,j),\, (i,j)\in D|\, a_k x(i,j)-a_j x(i,k)+a_i x(j,k)=0,\, (i,j,k)\in E\rangle.
$$

We say that $M$ is \emph{normal} if for each $p\in M$,  the $n-1$ vectors 
$v(i,p)$, $i\neq p$, are $R$-linearly independent. Observe that if no $a_p$, with $p\in M$, is a
zero divisor, then $M$ is normal. In particular, if $R$ is a domain, then 
$M$ automatically normal. We may now state the following result.

\begin{theorem}\label{rela} Suppose $M$ is normal.
Then (\ref{defrel}) are defining relations for $S$.
\end{theorem}

 \begin{proof} For a triple $(i,j,k)\in E$, we write $A_{ijk}\in M_n(R)$ 
for the strictly upper triangular matrix that has entries $a_k,\ -a_j$ and $a_i$ respectively in positions $(i,j),\ (i,k)$ and $(j,k)$, with all other entries equal to zero. We note, with reference to (\ref{defrel}), that
  $$
  \sum_{1\le p<q\le n} (A_{ijk})_{pq}v(p,q) = a_k v(i,j)-a_j v(i,k)+ a_i v(j,k) = 0.
$$
That the relations of (\ref{defrel}) define $S$ has the following equivalent interpretation: if $C\in M_n(R)$ is a strictly upper triangular matrix satisfying
$C_{ij}=0$ for $(i,j)\notin D$ and $\sum_{1\le i<j\le n} C_{ij}v(i,j)=0$, then $C$ is
an $R$-linear combination of the matrices $A_{ijk}$.

We now assume that $C$ is such a matrix. Since $I_M=R$, 
to show that $C$ is an $R$-linear combination of the $A_{ijk}$, it is sufficient to show that each 
$a_iC$, $i\in M$, is such a combination. To this end we fix 
$p\in M$ and observe that every entry of $a_pC$ is an $R$-multiple of 
$a_p$. For every pair $(i,j)$ such that $1\leq i<j\leq n$, $(i,j)\in D$, and $p\notin\{i,j\}$, 
the matrix $A_{[ijp]}$ has $\pm a_p$ in its $(i,j)$-position (where the square brackets mean that the indices are arranged in increasing order). There are only two other possible positions of nonzero entries in $A_{[ijp]}$, both having $p$ either as a row or column index. Thus $a_pC$ can be reduced by addition of $R$-multiples of the 
$A_{[ijp]}$ to a strictly upper triangular matrix $C'$ whose only possible nonzero entries are in row $p$ and column~$p$. This 
yields an expression for the zero vector in $R^n$ as an $R$-linear combination of the vectors $v(i,p)$, with $i\neq p$. 
By the normality of $M$, these vectors are linearly independent over $R$, 
so $C'=0$ and $a_pC$ is an $R$-linear combination of the matrices $A_{[ijp]}$. 
\end{proof}

\begin{note} The stated hypothesis cannot be dropped from Theorem \ref{rela}, as the case $n=3$,
$R=\Z/30\Z$, and $(a_1,a_2,a_3)=(6,10,15)$ shows. In this case, the relation $a_3v(1,2)-a_2v(1,3)=0$ cannot be deduced
from (\ref{defrel}) and the vectors $v(2,1)$ and $v(3,1)$ are linearly dependent over $R$.
\end{note}

\medskip

We next offer an alternative proof of Theorem \ref{rela}, valid when $R$ is a principal ideal domain. For this purpose,
set
$$
d=|D|={{m}\choose{1}}{{n-m}\choose{1}}+{{m}\choose{2}}\text{ and }e=|E|={{m}\choose{2}}{{n-m}\choose{1}}+{{m}\choose{3}},
$$
and let $A\in M_{d\times e}(R)$ be the matrix whose columns are the coordinates of the $y(i,j,k)$ relative to the $x(i,j)$ (in both cases, ordered lexicographically). For instance, if $n=4=m$, then 
$$
A=\left(\begin{array}{cccc} 
a_3 & a_4 & 0 & 0\\
-a_2 & 0 & a_4 & 0\\
0 & -a_2 & -a_3 & 0\\
a_1 & 0 & 0 & a_4\\
0 & a_1 & 0 & -a_3\\
0 & 0 & a_1 & a_2\end{array}\right).
$$

Our proof hinges on the Smith Normal Form of~$A$.

 \begin{theorem}\label{rela2} Suppose that $R$ is a principal ideal domain. Then (\ref{defrel}) are
defining relations for $S$.
\end{theorem}

 \begin{proof} We know that $Y\subseteq\ker\Lambda$. Hence, $\Lambda$ induces an epimorphism of 
$R$-modules $\Delta:X/Y\to S$, 
$$
\Delta(x+Y)=\Lambda(x),\quad x\in X.
$$
Here $\ker\Lambda=Y$ if and only if $\Delta$ is injective, and we proceed to show the latter. Let $S(A)$ be the Smith normal form of 
the matrix $A$ defined above, and set $r=\mathrm{rank}(A)$.
We claim that
\begin{equation}
\label{r}
r=d-(n-1)\text{ and }S(A)=\mathrm{diag}(\underbrace{1,\dots,1}_{r},0,\dots,0).
\end{equation}
Indeed, there exist nonzero elements $b_1,\dots,b_r\in R$ such that $b_1|\cdots|b_r$ and
\begin{equation}
\label{smith}
S(A)=\mathrm{diag}(\underbrace{b_1,\dots,b_r}_{r},0,\dots,0).
\end{equation}
Then (\ref{smith}) implies the existence of a basis $\{u_1,\dots,u_d\}$ of $X$
such that $\{b_1u_1,\dots,b_r u_r\}$ is a basis of~$Y$. Thus
$$
X/Y\cong R/Rb_1\oplus\cdots\oplus R/Rb_r\oplus R^{d-r}.
$$
Now $S\cong R^{n-1}$ and $S$ is an epimorphic image of $X/Y$. This readily implies that $d-r\geq n-1$, or
\begin{equation}
\label{t}
r\leq d-(n-1).
\end{equation}
On the other hand, for $1\leq i\leq r$, let $D_i$ be the greatest common divisor of the determinants of all $i\times i$ submatrices
of $A$. It is well-known that
\begin{equation}
\label{d}
D_1|\cdots|D_r\text{ and }
b_1=D_1, b_2=D_2/D_1,\dots,b_r=D_r/D_{r-1}.
\end{equation}
Set
$$
z=d-(n-1)={{m-1}\choose{1}}{{n-m}\choose{1}}+{{m-1}\choose{2}}.
$$
Since $I_M=R$, by virtue of (\ref{t}) and (\ref{d}), in order to show (\ref{r}),
it suffices to show that given any $i\in M$, there is a $z\times z$ diagonal submatrix, say $A(i)$, of $A$ with diagonal entries equal to $\pm a_i$. This is easy: the columns of $A(i)$ are given by all 3-subsets 
$\{i,j,k\}$ of $N$ such that $\{j,k\}\cap M\neq\emptyset$; each of these columns contains at most three nonzero entries, in the rows of the three 2-subsets 
$\{i,j\}$, $\{i,k\}$, and
$\{j,k\}$ of $\{i,j,k\}$, and we select the row corresponding to $\{j,k\}$. Thus, given any 2-subset $\{j,k\}$ of $N$ such that
$i\notin\{j,k\}$ and $\{j,k\}\cap M\neq\emptyset$, 
the matrix $A(i)$ has at most one nonzero entry in column 
$\{j,k\}$, namely $\pm a_i$, in row $\{j,k\}$ (the sign depends on how $\{i,j,k\}$ is ordered).

This proves  (\ref{r}). We next indicate how to use (\ref{r}) to show that $\Delta$ is injective. As mentioned above, there is a basis $\{u_1,\dots,u_d\}$ of $X$ 
such that $\{u_1,\dots,u_r\}$ is a basis of $Y$.
Thus $X/Y$ is a free $R$-module of rank $d-r=n-1$. Hence, there is an isomorphism of $R$-modules 
$\Omega:S\to X/Y$. This yields
the epimorphism of $R$-modules $\Delta\Omega:S\to S$. As a nonzero free $R$-module of finite rank, $S$ is not isomorphic to any of 
its proper quotients.  Thus $\Delta\Omega$ is injective, and therefore so is $\Delta=\Delta\Omega\Omega^{-1}$.
\end{proof}

\section{Structure of $S/S_i$}

In this section we determine the $R$-module structure of $S/S_i$. We note that $S/S_i$ is a module over the ring $R/Ra_i$, since $a_iS\subseteq S_i$. Under suitable hypotheses on $R$, we find that $S/S_i\cong (R/Ra_i)^{n-2}$. In particular, this statement holds if $R$ is a principal ideal domain.

For convenience, we assume in our analysis that $i=1$. We write $\overline{R}$ for the quotient ring $R/Ra_1$, and for $a\in R$ 
we write $\bar{a}$ for the element $a+Ra_1$ of $\overline{R}$. We begin by showing that $S/S_1$ is isomorphic to a submodule of 
$\overline{R}^{n-1}$. We define $\theta :S\to \overline{R}^{n-1}$ by $\theta (x_1,\dots ,x_n)=(\bar{x_2},\dots ,\bar{x_n})$.

\begin{lemma}\label{theta}
The kernel of $\theta$ is $S_1$.
\end{lemma}

\begin{proof}
It is immediate that $v(1,j)\in\ker\theta$ for $j\ge 2$, hence $S_1\subseteq\ker\theta$. 
On the other hand, suppose that $s\in\ker\theta$. Then for $i\ge 2$, the $i$th component of $s$ is $c_ia_1$, where $c_i\in R$. We write 
$s' = \sum_{j=2}^n c_j v(1,j)$. Then $s'\in S_1$ and $s-s'$ is an element of $\ker\theta$ whose first entry is some $b\in R$ with 
$a_1b=0$, and whose subsequent entries are all zero. There exist elements $d_1,\dots ,d_n$ of $R$ for which 
$\sum_{j=1}^n d_ja_j=1$, so $b=\sum_{j=1}^n d_ja_jb = \sum_{j=2}^n d_jba_j$. Then
$$
-\sum_{j=2}^n d_jb v(1,j) = (b,-d_2ba_1,\dots ,-d_nba_1) = (b,0,\dots ,0). 
$$
Thus $s\in S_1$ and $\ker\theta = S_1$. 
\end{proof}

The image of $\theta$ is the submodule $\overline{S}$ of $\overline{R}^{n-1}$ consisting of all $(\bar{x_2},\dots ,\bar{x_n})$ for which 
$$
\bar{a_2}\bar{x_2}+\dots +\bar{a_n}\bar{x_n}=0.
$$

Note that $(\bar{a_2},\dots ,\bar{a_n})$ is a unimodular vector in $\bar{R}^{n-1}$, since $(a_1,a_2,\dots ,a_n)$ is unimodular 
in~$R^n$. We wish to identify conditions under which $\overline{S}$ is a free $\overline{R}$-module of rank $n-2$. We assume
for the remainder of the paper that $R$ is commutative. According to \cite{L}, $R$ is said to be a $K$-Hermite ring if for every 
$m\geq 1$ and $(b_1,\dots,b_m)\in R^m$, there exist $Q\in\mathrm{GL}_m(R)$ and $d\in R$ such that $(b_1,\dots,b_m)Q = (d,0,\dots,0)$.
Note that this implies $Rb_1+\cdots+Rb_m=Rd$. If $R$ is a principal ideal domain and the Smith Normal form of $(b_1,\dots,b_m)$
is $(d,0,\dots,0)$, then $(b_1,\dots,b_m)Q = (d,0,\dots,0)$ for some $Q\in\mathrm{GL}_m(R)$, so $R$ is a $K$-Hermite ring.
A detailed discussion of $K$-Hermite rings can be found in  \cite[Chapter 1]{L}.

\begin{theorem}\label{S/S_1}
  If $R$ is a $K$-Hermite ring and $i\in\{1,\dots ,n\}$, then $S/S_i$ is isomorphic to $(R/Ra_i)^{n-2}$.
\end{theorem}

\begin{proof} We may assume without loss that $i=1$. There are $d\in R$ and $Q\in\mathrm{GL}(n-1,R)$ such that 
$(a_2,\dots,a_n)= (d,0,\dots,0)Q$. Thus $(a_2,\dots,a_n)$ is the first row of a matrix $A\in M_{n-1}(R)$ whose determinant is $du$, with $u\in U(R)$. Suppose first $a_1\in U(R)$. Then $S=S_1$. Suppose next $a_1\notin U(R)$ and let 
$\overline{A}\in M_{n-1}(\overline{R})$ be the matrix corresponding to  $A$. Since 
$Ra_1+Rd=R$, the determinant $\bar{d}$ of $\overline{A}$ is a unit in $\overline{R}$, and $\overline{A}$ has an  inverse 
$B\in M_{n-1}(\overline{R})$. Thus columns 2 through $n-1$ of $B$ yield a basis of $\overline{S}$, 
whence $S/S_1\cong \overline{R}^{n-2}$ by Lemma \ref{theta}.
\end{proof}

\section{Constructing an $R$-basis of $S$}

We suppose for the remainder of the paper that $R$ is an integral domain with field of fractions~$F$,
and let $T$ be the subspace
of $F^n$ of all solutions to (\ref{ecu}). 

Since $\gcd(a_1,\dots,a_n)=1$, some $a_i\neq 0$ and we assume for
notational convenience that $a_1\neq 0$. Then the following vectors form an $F$-basis of $T$:
\begin{equation}
\label{wi}
w(1,2)=(-a_2/a_1, 1, 0,\dots,0),w(1,3)=(-a_3/a_1, 0, 1, 0,\dots,0),\dots,
w(1,n)=(-a_n/a_1, 0,\dots,0,1),
\end{equation}
and we let $W$ be the $R$-span of the vectors $w(1,2),\dots,w(1,n)$. Given $\alpha_2,\dots,\alpha_n\in F$, we have
\begin{equation}
\label{inS}
\alpha_2w(1,2)+\cdots+\alpha_n w(1,n)\in S\Leftrightarrow \alpha_2,\dots,\alpha_n\in R\text{ and }
a_2\alpha_2+\cdots+a_n\alpha_n\equiv 0\mod a_1.
\end{equation}
This implies, in particular, that $S\subseteq W$.

\begin{theorem} 
\label{ws} Suppose $a_1\neq 0$ and let $W$ be the $R$-span of the vectors (\ref{wi}). Then $W/S\cong R/Ra_1$.
\end{theorem}

\begin{proof}
We define a mapping $\theta :W\to R/Ra_1$ by $\theta (t)=a_1t_1+Ra_1$, for $t=(t_1,\dots ,t_n)\in W$. Here $t_1\in F$ and $a_1t_1\in R$.
The kernel of $\theta$ consists of those $t$ for which $t_1\in R$, which is exactly~$S$. Finally, there exist elements $c_1,\dots ,c_n\in R$ with
$
c_1a_1+c_2a_2+\dots +c_na_n = 1.
$
The vector $-(c_2w_2+\dots +c_nw_n)$ has $\frac{1}{a_1}(c_2a_2+\dots +c_na_n)$ as its first entry, and its image under $\theta$ is
$$
(c_2a_2+\dots +c_na_n)+Ra_1 = (1-c_1a_1)+Ra_1 = 1+Ra_1. 
$$
Since the image of $\theta$ is an $R$-module containing $1+Ra_1$, it is equal to $R/Ra_1$, and we conclude that $W/S\cong R/Ra_1$. 
\end{proof}

We assume for the remainder of the paper that $R$ is a principal ideal domain. Then $S$ is a free $R$-module of rank $n-1$ and
the goal of this section is to construct an $R$-basis of $S$. 

\begin{cor}\label{quot2} There is a basis $\{u_1,\dots,u_{n-2},u_{n-1}\}$ of $W$ such that
$\{u_1,\dots,u_{n-2},a_1 u_{n-1}\}$ is a basis of $S$ (and therefore $S/S_1=S/a_1 W\cong (R/a_1 R)^{n-2}$).
\end{cor}

\begin{proof} Since $S\subseteq W$ are free $R$-modules of rank $n-1$, there are bases $\{u_1,\dots,u_{n-1}\}$
and $\{v_1,\dots,v_{n-1}\}$ of $W$ and $S$, respectively, as well as nonzero elements 
$d_1,\dots,d_{n-1}$, unique up to multiplication by units,
such that $v_i=d_i u_i$ and $d_1|\cdots|d_{n-1}$. Thus, by Theorem \ref{ws}, we have
$$
R/Rd_1\oplus\cdots\oplus R/Rd_{n-1}\cong W/S\cong R/Ra_1.
$$
Since the invariant factors of $W/S$ are uniquely determined, it follows that $$d_1=\cdots=d_{n-2}=1,\; d_{n-1}=a_1.$$
Thus $\{u_1,\dots,u_{n-2},a_1 u_{n-1}\}$ is a basis of $S$.
\end{proof}

Recall that $M$ stands for a subset of $N$ such that $I_M=R$.
In practice, it is convenient to select $M$ to be minimal subject to this property, but in our theory this is not relevant
and we do not assume it.
By reordering the variables in (\ref{ecu}), and for notational convenience, we will assume that 
$M=\{a_1,\dots,a_m\}$, where $m=|M|$. If $m=1$ then $\{v(1,2),\dots,v(1,n)\}$
is already a basis of $S$. Suppose next $m=2$ (since $\gcd(a_1,\dots,a_n)=1$, provided some $a_i$ is a prime power,
we may certainly choose $M$ so that $|M|=2$ and $I_M=R$). Since $\gcd(a_1,a_2)=1$ in this case, we can find $b\in R$ so that 
$a_2b\equiv 1\mod a_1$. Set $c_i=-a_ib\in R$, $3\leq i\leq n$. Then, as shown below, the vectors
\begin{equation}\label{zi}
z_2=a_1w(1,2), z_3=c_3w(1,2)+w(1,3),\dots, z_n=c_n w(1,2)+w(1,n)
\end{equation}
form a basis of $S$. The general case is an extension of this one. We require the following result.

\begin{lemma}\label{Z0} Let $X\leq Y\leq Z$ be free $R$-modules of finite rank $r\geq 1$, with bases 
$\{x_1,\dots,x_r\}$, $\{y_1,\dots,y_r\}$, and $\{z_1,\dots,z_r\}$, respectively. Let $A\in M_r(R)$ (resp. 
$B\in M_r(R)$) be the matrix whose $j$th column is formed by the coordinates of $x_j$ (resp. $y_j$) relative to 
$\{z_1,\dots,z_r\}$, and suppose that $|A|=|B|$. Then $X=Y$.
 \end{lemma}

\begin{proof} Let $C\in M_r(R)$ be the matrix whose $j$th column is formed by the coordinates of $x_j$ relative to 
$\{y_1,\dots, y_r\}$. Then $A=BC$, so $|B|=|A|=|B||C|$. The columns of $B$ are linearly independent, so $|B|\neq 0$,
and hence $|C|=1$. Thus $C\in\GL_r(R)$, and therefore all $y_j$ are in $X$.
\end{proof}

\begin{cor}\label{Z} Let $Z$ be an $R$-submodule of $S$ having a basis
$\{z_2,\dots,z_n\}$ such that the determinant of the matrix whose $j$th column is formed
by the coordinates $z_j$ relative to $\{w(1,2),\dots,w(1,n)\}$ is equal to $a_1$. Then $S=Z$ (and hence $\{z_2,\dots,z_n\}$ is
is a basis of $S$).
\end{cor}

\begin{proof} This follows from Corollary \ref{quot2} and Lemma \ref{Z0}.
\end{proof}

Corollary \ref{Z} readily implies that the vectors (\ref{zi}) form a basis of $S$. Indeed, we see from (\ref{inS})
all vectors the (\ref{zi}) are in $S$, so their span, say $Z$, is contained in $S$.
Moreover, the corresponding matrix has determinant $a_1$, as required. 

We now state our main theorem on the direct construction of a basis of $S$ from the elements
$a_1,\dots,a_n$. We write $(r_1,\dots,r_t)$
for the greatest common divisor of the elements $r_1,\dots,r_t\in R$.

\begin{theorem}\label{basis} Suppose $M=\{a_1,\dots,a_m\}$ satisfies $I_M=R$, where $a_1\neq 0$ and $1<m\leq n$.
Then we can construct matrices $X\in M_{m-1}(R)$, $Y\in M_{m-1,n-m}(R)$ and 
\begin{equation}
\label{matriz}
A=\left(\begin{array}{cc} X & Y\\ 0 & I\end{array}\right)\in M_{n-1}(R),
\end{equation}
such that $X$ is upper triangular with diagonal entries
\begin{equation}\label{di}
a_1/(a_1,a_2), (a_1,a_2)/(a_1,a_2,a_3),\dots,(a_1,\dots,a_{m-2})/(a_1,\dots,a_{m-1}), (a_1,\dots,a_{m-1}),
\end{equation}
and the columns of $A$ are the coordinates relative to $w(1,2),\dots,w(1,n)$ of a basis $z_2,\dots,z_n$ of $S$.
\end{theorem}

\begin{proof} By (\ref{inS}) and Corollary \ref{Z}, it suffices to verify that the off-diagonal entries of $X$ and the matrix
$Y$ can be chosen so that each column $(\alpha_2,\dots,\alpha_n)$ of $A$ satisfies
\begin{equation}\label{al}
\alpha_2 a_2+\cdots+\alpha_n a_n\equiv 0\mod a_1.
\end{equation}
Let us first verify (\ref{al}) for the first $m-1$ columns of $A$, that is, for the columns of $X$.
Since 
$$
a_2a_1/(a_1,a_2)\equiv 0\mod a_1,
$$
the result is clear for the first column of $X$. Suppose next $2\leq i\leq m-1$. We wish to verify (\ref{al}) for the $i$th column 
of~$X$. Using $(a_1,\dots,a_m)=1$ when $i=m-1$, this verification translates into
\begin{equation}\label{x}
a_2 X_{1,i}+a_3 X_{2,i}+\cdots+a_i X_{i-1,i}+a_{i+1}(a_1,\dots,a_{i})/(a_1,\dots,a_{i+1})\equiv 0\mod a_1.
\end{equation}
As $(a_1,\dots,a_{i})$ divides $a_{i+1}(a_1,\dots,a_{i})/(a_1,\dots,a_{i+1})$, we can certainly complete the $i$th column of~$X$ so that (\ref{x}) is satisfied.

As for the remaining columns of $A$ (which only exist when $m<n$), we need to be able to find one solution to
$$
a_2 Y_{1,j}+\cdots+a_m Y_{m-1,j}\equiv -a_{m+j}\mod a_1,\quad 1\leq j\leq n-m,
$$
which is certainly possible since $\gcd(a_1,a_2,\dots,a_m)=1$.
\end{proof}

Let $a$ be the column vector in $R^n$ with entries $a_1,\dots,a_n$. Since $(a_1,\dots,a_n)=1$, the Smith Normal Form of $a$
is the first canonical vector of $R^n$, say $b$. Thus, there is $E\in\GL_n(R)$ such that $Ea=b$. This readily
implies that rows $2,\dots,n$ of $E$ form a basis of $S$. Theorem \ref{basis} is so closely related to $a_1,\dots,a_n$
that it allows us in \S\ref{S/U_i} to elucidate the structure of certain quotients of $S$.
We found no way of doing this by means of $E$.


\section{Structure of $S/U_i$}\label{S/U_i}

For $1\leq i\neq j\leq n$, we set $u(i,j)=v(i,j)$ if $a_i$ and $a_j$ are both zero,
and $u(i,j)=v(i,j)/(a_i,a_j)$ otherwise.
Since the vectors $v(i,j)$, $i<j$, span $S$, the vectors $u(i,j)$, $i<j$, also span $S$. 
For $1\leq i\leq n$, we let $U_i$ be the submodule of $S$ spanned by all vectors $u(i,j)$ with
$j\neq i$. It is clear that $S_i\subseteq U_i\subseteq S$.

In this section we use
Theorem \ref{basis} in order to  understand the structure of the $R$-module $S/U_i$. This is considerably
more difficult than our prior  study of $S/S_i$. 

For convenience we assume $i=1$. If $a_1=0$ we readily verify that $U_1=S_1$, whence $S/U_1\cong R^{n-2}$
by Theorem \ref{S/S_1}. On the other hand, if $n=2$ we easily see that 
$S=U_1$. Thus, we may suppose that $a_1\neq 0$ and $n>2$, and we do so
for the remainder of this section.

Let $D\in M_{n-1}(R)$ be the matrix whose columns are the coordinates of $u(1,2),\dots,u(1,n)$
relative to $w(1,2),\dots,w(1,n)$. Thus,
\begin{equation}
\label{mad}
D=\mathrm{diag}(a_1/(a_1,a_2),\dots,a_1/(a_1,a_n)).
\end{equation}
Let $A\in M_{n-1}(R)$ be as in (\ref{matriz}) and let $\{z_2,\dots,z_{n}\}$ be the corresponding basis of $S$.
We write $C\in M_{n-1}(R)$ for the matrix whose columns are the coordinates of $u(1,2),\dots,u(1,n)$
relative to the basis $\{z_2,\dots,z_n\}$ of $S$. Then from general principles, we have 
$$
C=A^{-1}D.
$$
Since $|A|=a_1$, we infer
$$
|C|=\frac{a_1^{n-2}}{(a_1,a_{2})\cdots (a_1,a_{n})}.
$$

Let $A$ be as in (\ref{matriz}) when $M=N$, so that $A=X$ is upper triangular.  Then $C$
is an upper triangular matrix, with diagonal entries $d_1,\dots,d_{n-1}$, where $d_1=1$ and
\begin{equation}\label{dis}
d_i=\frac{a_1(a_1,\dots,a_{i+1})}{(a_1,a_{i+1})(a_1,\dots,a_{i})}=
\frac{a_1}{\mathrm{lcm}[(a_1,a_{i+1}),(a_1,\dots,a_{i})]},\quad 2\leq i\leq n-1.
\end{equation}

From Theorem \ref{S/S_1}, we have $S/S_1\cong R^{n-2}/(Ra_1)^{n-2}$. Since $S/U_1\cong (S/S_1)/(U_1/S_1)$, we infer 
that $S/U_1\cong R^{n-2}/Q$, where $Q$ is a submodule of $R^{n-2}$ containing $(Ra_1)^{n-2}$. There is a basis
$\{u_1,\dots,u_{n-2}\}$ of $R^{n-2}$ and nonzero $f_1,\dots,f_{n-2}$ in $R$, all factors of $a_1$,
such that $\{f_1u_1,\dots,f_{n-2}u_{n-2}\}$ is a basis of $Q$. It follows that
\begin{equation}\label{fi}
f_1\cdots f_{n-2}=\frac{a_1^{n-2}}{(a_1,a_{2})\cdots (a_1,a_{n})},
\end{equation}
up to multiplication by units, and
\begin{equation}\label{fi2}
S/U_1\cong R/Rf_1\oplus\cdots\oplus R/Rf_{n-2}.
\end{equation}
In particular, if $n=3$, we have
$$
f_1=\frac{a_1}{(a_1,a_2)(a_1,a_3)}=d_2,\text{ and }S/U_1\cong R/Rf_1.
$$
\begin{cor} We have $S=U_1$ if and only if $a_1^{n-2}=(a_1,a_{2})\cdots (a_1,a_{n})$, up to multiplication by units.
In particular, if $p_1,\dots,p_n\in R$ are nonassociate primes and $a_i=\underset{j\neq i}\Pi p_j^{e_{i,j}}$,
$1\leq i\leq n$, where $e_{1,j}\leq e_{i,j}$ whenever $1,i,j$ are distinct, then $S=U_1$ and
$\{u(1,2),\dots,u(1,n)\}$ is a basis of $S$.
\end{cor}

\begin{proof} This follows immediately from (\ref{fi}) and (\ref{fi2}).
\end{proof}

In Theorem \ref{S/U_1}, we will show that $(f_1,\dots ,f_{n-2})=(d_2,\dots ,d_{n-1})$ is a solution of (\ref{fi2}) in general. We will do this by identifying the $p$-elementary divisors of $S/U_1$, for each prime $p$ in $R$. The following theorem is the main technical element needed for this step. 

\begin{theorem} Suppose $a_1\neq 0$ and $n>2$. 
Let $A$ be as in (\ref{matriz}) when $M=N$, and let $C=A^{-1}D$, with $D$ as in (\ref{mad}). Then
\begin{equation}
\label{Cij}
C_{i,j}\in R C_{i,i}+R C_{j,j},\quad 1\leq i<j\leq n-1.
\end{equation}
\end{theorem}

\begin{proof} We have $A = \left( x_{i,j}\right)$, where 
	\begin{equation}\label{xii}
		x_{i,i} = \frac{(a_1, \ldots, a_i)}{(a_1, \ldots , a_{i+1})}, \; 1 \leq i < n,
	\end{equation}
	and  
	\begin{equation}\label{xij}
		a_2 x_{1,j} + \cdots + a_j x_{j-1,j} + a_{j+1} \frac{(a_1, \ldots, a_j)}{(a_1, \ldots , a_{j+1})}\equiv 0 \mod a_1, \; 1 \leq j <n.
	\end{equation}
Moreover, the diagonal entries of $C$ are clearly given by 
$$C_{i,i} = \frac{a_1 (a_1, \ldots, a_{i+1})}{(a_1, \ldots, a_i)(a_1, a_{i+1})}, \; 1 \leq i <n.$$
Regarding the off-diagonal entries of $C$, let $E$ be the adjoint of $A$.
Thus $E$ is upper triangular~and
$$
E_{i,j}=(-1)^{i+j}|A(j,i)|,\quad 1\leq i<j<n,
$$
where $A(j,i)$ is the submatrix obtained from $A$ by deleting its $j$th row and $i$th column.
Hence
$$AE=|A|I_{n-1}=a_1I_{n-1},\text{ so }C=\frac{ED}{a_1}.$$
The matrix $A(j,i)$ is block upper triangular. The first $i-1$ and the last $n-1-j$ diagonal blocks have size $1$
and are simply the corresponding diagonal entries of $A$,  while the middle diagonal block, say $N_{i,j}$, 
has size $j-i$ and is the submatrix of $A$ determined by rows $i,\dots,j-1$ and columns $i+1,\dots,j$, whose
determinant will be denoted by $M_{i,j}$. Computing $|A(j,i)|$ by making use of these diagonal blocks, we 
see that for $1\leq i <j<n$, we have
\begin{equation}\label{relcij}
	\begin{split}
		C_{i,j} & =  \frac{a_1}{(a_1,a_{j+1})} \frac{E(i,j)}{a_1}\\
		 & = (-1)^{i+j}\frac{a_1}{(a_1,a_{j+1})} \frac{1}{a_1}x_{1,1} \cdots x_{i-1,i-1} x_{j+1,j+1} \cdots x_{n-1,n-1} M_{i,j}\\
		 & = (-1)^{i+j}\frac{a_1}{(a_1,a_{j+1})} \frac{1}{a_1} \frac{a_1}{(a_1,a_2)} \cdots \frac{(a_1,\ldots, a_{i-1})}{(a_1,\ldots, a_i)} \frac{(a_1, \ldots, a_{j+1})}{(a_1, \ldots,a_{j+2})} \cdots \frac{(a_1, \ldots,a_{n-1})}{(a_1,\ldots,a_n)}M_{i,j}\\
		 & = (-1)^{i+j}\frac{a_1 (a_1,\ldots , a_{j+1})}{(a_1, \ldots , a_i)(a_1,a_{j+1})}  M_{i,j}.
		\end{split} 
	\end{equation}
Thus, to prove (\ref{Cij}) we need find to $u,v\in R$ such that 
	$$ 
	\frac{a_1(a_1, \ldots , a_{i+1})}{(a_1, \ldots , a_i)(a_1, a_{i+1})} u + 
	\frac{a_1(a_1, \ldots , a_{j+1})}{(a_1, \ldots , a_j)(a_1, a_{j+1})} v = 
	\frac{a_1(a_1, \ldots , a_{j+1})}{(a_1, \ldots , a_i)(a_1, a_{j+1})} M_{i,j},  
	$$ 
	which is equivalent to
	\begin{multline*}
		(a_1, \ldots, a_{i+1}) ( a_1, \ldots, a_j) (a_1, a_{j+1}) u + (a_1, \ldots, a_{j+1}) (a_1, \ldots, a_i) (a_1, a_{i+1}) v \\
		= (a_1, \ldots, a_{j+1}) (a_1, \dots, a_j) (a_1, a_{i+1}) M_{i,j},
	\end{multline*}
which is equivalent to
\begin{multline*}
 \frac{( a_1, \ldots, a_j)}{(a_1, \ldots, a_{j+1})} \frac{(a_1, a_{j+1})}{(a_1, \ldots, a_{j+1})} u +  \frac{(a_1, \ldots, a_i)}{(a_1, \ldots, a_{j+1})} \frac{(a_1, a_{i+1})}{(a_1, \ldots, a_{i+1})} v \\
	= \frac{(a_1, \dots, a_j)}{(a_1, \ldots, a_{j+1})} \frac{(a_1, a_{i+1})}{(a_1, \ldots, a_{i+1})} M_{i,j}.
\end{multline*}
Such elements $u$ and $v$ exist if and only if
$$d=\left( \frac{( a_1, \ldots, a_j)}{(a_1, \ldots, a_{j+1})} \frac{(a_1, a_{j+1})}{(a_1, \ldots, a_{j+1})} , 
\frac{(a_1, \ldots, a_i)}{(a_1, \ldots, a_{j+1})} \frac{(a_1, a_{i+1})}{(a_1, \ldots, a_{i+1})} \right)
$$
is a factor of
$$
\frac{(a_1, \dots, a_j)}{(a_1, \ldots, a_{j+1})} \frac{(a_1, a_{i+1})}{(a_1, \ldots, a_{i+1})} M_{i,j}.
$$
Since $\frac{( a_1, \ldots, a_j)}{(a_1, \ldots, a_{j+1})}$ and $\frac{(a_1, a_{j+1})}{(a_1, \ldots, a_{j+1})}$ are relatively prime, we have
$$
d= \left( \frac{( a_1, \ldots, a_j)}{(a_1, \ldots, a_{j+1})}  , \frac{(a_1, \ldots, a_i)}{(a_1, \ldots, a_{j+1})} 
\frac{(a_1, a_{i+1})}{(a_1, \ldots, a_{i+1})} \right)
\left( \frac{(a_1, a_{j+1})}{(a_1, \ldots, a_{j+1})} , \frac{(a_1, \ldots, a_i)}{(a_1, \ldots, a_{j+1})} 
\frac{(a_1, a_{i+1})}{(a_1, \ldots, a_{i+1})} \right).
$$
The first factor is equal to $\frac{( a_1, \ldots, a_j)}{(a_1, \ldots, a_{j+1})},$ 
because $\frac{( a_1, \ldots, a_j)}{(a_1, \ldots, a_{j+1})}$ divides $\frac{(a_1, \ldots, a_i)}{(a_1, \ldots, a_{j+1})}$,
while the second factor is equal to 
$$
\frac{(a_1, \ldots, a_i, a_{j+1})}{(a_1, \ldots, a_{j+1})} \left( \frac{(a_1, a_{j+1})}{(a_1, \ldots, a_i, a_{j+1})} , \frac{(a_1, \ldots, a_i)}{(a_1, \ldots,a_i, a_{j+1})} \frac{(a_1, a_{i+1})}{(a_1, \ldots, a_{i+1})} \right)  $$
$$= \frac{(a_1, \ldots, a_i, a_{j+1})}{(a_1, \ldots, a_{j+1})} \left( \frac{(a_1, a_{j+1})}{(a_1, \ldots, a_i, a_{j+1})} ,  \frac{(a_1, a_{i+1})}{(a_1, \ldots, a_{i+1})} \right),$$ 
which is a factor of 	$$\frac{(a_1, \ldots, a_i, a_{j+1})}{(a_1, \ldots, a_{j+1})} \frac{(a_1, a_{i+1})}{(a_1, \ldots, a_{i+1})}.$$ 
Thus  $d$ divides 
$$\frac{( a_1, \ldots, a_j)}{(a_1, \ldots, a_{j+1})}  \frac{(a_1, a_{i+1})}{(a_1, \ldots, a_{i+1})} 
\frac{(a_1, \ldots, a_i, a_{j+1})}{(a_1, \ldots, a_{j+1})}.
$$

The theorem will be proved if $ \frac{(a_1, \ldots, a_i, a_{j+1})}{(a_1, \ldots, a_{j+1})}$ divides $ M_{i,j},$ for all 
$1 \leq i < j \leq n-1$. We will show this by induction on $j-i$. Suppose first that $j-i=1$. We wish to show that
$ \frac{(a_1, \ldots, a_i, a_{i+2})}{(a_1, \ldots, a_{i+2})}$ divides $M_{i,i+1} = x_{i,i+1}$. Taking $j=i+1$ in
(\ref{xij}), we obtain
$$a_2 x_{1,i+1} + \cdots + a_{i+1} x_{i,i+1} + a_{i+2} \frac{(a_1, \ldots, a_{i+1})}{(a_1, \ldots , a_{i+2})}\equiv 0 \mod a_1,$$ 
 which implies
\begin{multline*}
	\frac{a_2}{(a_1, \ldots, a_{i+2})} x_{1,i+1} + \cdots + \frac{a_{i}}{(a_1, \ldots, a_{i+2})} x_{i-1,i+1} + \frac{a_{i+1}}{(a_1, \ldots, a_{i+2})} x_{i,i+1} \\
	+ \frac{a_{i+2}}{(a_1, \ldots, a_{i+2})} \frac{(a_1, \ldots, a_{i+1})}{(a_1, \ldots , a_{i+2})}\equiv 0\mod 
	\frac{a_1}{(a_1, \ldots, a_{i+2})}.
\end{multline*}
Each term on the left hand side except for $\frac{a_{i+1}}{(a_1, \ldots, a_{i+2})} x_{i,i+1}$ is a multiple of 
$\frac{(a_1,\ldots, a_i, a_{i+2})}{(a_1, \ldots, a_{i+2})}$. Since
$\frac{a_1}{(a_1, \ldots, a_{i+2})}$ is also a multiple of  $\frac{(a_1,\ldots, a_i, a_{i+2})}{(a_1, \ldots, a_{i+2})}$,
we infer that $\frac{(a_1,\ldots, a_i, a_{i+2})}{(a_1, \ldots, a_{i+2})}$ divides $\frac{a_{i+1}}{(a_1, \ldots, a_{i+2})} x_{i,i+1}$.
But $\frac{a_{i+1}}{(a_1, \ldots, a_{i+2})}$  and  $\frac{(a_1,\ldots, a_i, a_{i+2})}{(a_1, \ldots, a_{i+2})}$  are relatively prime, so    $\frac{(a_1,\ldots, a_i, a_{i+2})}{(a_1, \ldots, a_{i+2})}$ divides $x_{i,i+1}$. This proves the case $j-i=1$. 

Let $1<k\leq n-2$ and suppose that $ \frac{(a_1, \ldots, a_i, a_{j+1})}{(a_1, \ldots, a_{j+1})}$ divides $ M_{i,j}$  when $j-i <k$.
We will show that $ \frac{(a_1, \ldots, a_i, a_{j+1})}{(a_1, \ldots, a_{j+1})}$ divides $ M_{i,j}$  when $j-i =k$.

We begin by writing $M_{i,j}$ in terms of $M_{i+r,j}$, $1\leq r\leq j-i-1$, in order to be able to invoke the inductive hypothesis.
We require additional notation for this purpose. Given subsets $P,Q$ of $\{1,\dots,n-1\}$, let $A_{P,Q}$ be the submatrix
of $A$ with rows and columns determined by $P$ and $Q$, respectively. With this notation, we have 
$N_{i,j}=A_{\{i,\dots,j-1\},\{i+1,\dots,j\}}$. Notice that this matrix is {\em almost upper triangular}, in the
sense that its subdiagonals below the first one are all equal to zero. Notice that the first subdiagonal of 
$A_{\{i,\dots,j-1\},\{i+1,\dots,j\}}$ consists of the diagonal entries $x_{i+1,i+1},\dots, x_{j-1,j-1}$ of $A$. 
In particular, all entries in the first column of 
$A_{\{i,\dots,j-1\},\{i+1,\dots,j\}}$ below the second row are equal to zero. Now $M_{i,j}$ is the determinant of this matrix,
so expanding $M_{i,j}$ by the first column, we obtain
$$
M_{i,j}=x_{i,i+1} M_{i+1,j} - x_{i+1,i+1}|A_{\{i,i+2,\dots,j-1\},\{i+2,\dots,j\}}|.
$$
Here $A_{\{i,i+2,\dots,j-1\},\{i+2,\dots,j\}}$ is again almost upper triangular, with entries $x_{i+2,i+2},\dots, x_{j-1,j-1}$
in its first subdiagonal. Expanding $|A_{\{i,i+2,\dots,j-1\},\{i+2,\dots,j\}}|$ along its first column yields
$$
|A_{\{i,i+2,\dots,j-1\},\{i+2,\dots,j\}}|=x_{i,i+2} M_{i+2,j} - x_{i+2,i+2}|A_{\{i,i+3,\dots,j-1\},\{i+3,\dots,j\}}|.
$$
Continuing this process, we eventually arrive at
\begin{multline}\label{expans}
	M_{i,j} = x_{i,i+1} M_{i+1,j} - x_{i+1,i+1} x_{i,i+2} M_{i+2,j} \\
	+ \cdots\\ 
	+ (-1)^{j-i} x_{i+1,i+1} \cdots x_{j-2,j-2} x_{i,j-1} M_{j-1,j} \\
	+ (-1)^{j-i+1} x_{i+1,i+1} \cdots x_{j-1,j-1} x_{i,j}.
\end{multline}
We next use (\ref{expans}) to prove the following intermediate step:
\begin{equation}
\label{div1}
\frac{(a_1, \ldots, a_{i+1}, a_{j+1})}{(a_1, \ldots, a_{j+1})}|M_{i,j}.
\end{equation}
It suffices to verify that each term on the right hand side of (\ref{expans}) is a multiple of 
$\frac{(a_1, \ldots, a_{i+1}, a_{j+1})}{(a_1, \ldots, a_{j+1})}$.

We start this verification with the first term. Since $j-(i+1)< k$, by induction hypothesis, we  have
$\frac{(a_1, \dots, a_{i+1},a_{j+1})}{(a_1, \ldots, a_{j+1})} | M_{i+1,j}$, as required. We continue our verification with the middle terms:
$$(-1)^{r+1}x_{i+1,i+1} \cdots x_{i+r-1,i+r-1} x_{i,i+r} M_{i+r,j},\quad 2 \leq r \leq j-i-1.$$
We have $\left( (a_1,\ldots, a_{i+1}, a_{j+1}),(a_1, \ldots, a_{i+r}) \right) = (a_1, \ldots, a_{i+r}, a_{j+1})$, so
$$\frac{(a_1,\ldots, a_{i+1}, a_{j+1})(a_1,\ldots, a_{i+r})}{(a_1, \ldots, a_{i+r}, a_{j+1})} = 
\text{lcm}\left( (a_1,\ldots, a_{i+1}, a_{j+1}),(a_1,\ldots, a_{i+r})\right).$$
Since $(a_1,\ldots, a_{i+1}, a_{j+1}) | (a_1,\ldots, a_{i+1})$ and $(a_1,\ldots, a_{i+r}) | (a_1,\ldots, a_{i+1})$, we obtain
$$\text{lcm}\left( (a_1,\ldots, a_{i+1}, a_{j+1}),(a_1,\ldots, a_{i+r})\right) |  (a_1,\ldots, a_{i+1}),$$ which yields
$$\frac{(a_1,\ldots, a_{i+1}, a_{j+1})(a_1,\ldots, a_{i+r})}{(a_1, \ldots, a_{i+r}, a_{j+1})} | (a_1,\ldots, a_{i+1}),$$
and therefore
$$\frac{(a_1,\ldots, a_{i+1}, a_{j+1})}{(a_1, \ldots, a_{i+r}, a_{j+1})} | \frac{(a_1, \ldots a_{i+1})}{(a_1,\ldots, a_{i+r})} .$$
But, $$x_{i+1,i+1} \cdots x_{i+r-1,i+r-1} = \frac{(a_1, \ldots, a_{i+1})}{(a_1, \ldots, a_{i+2})} \cdots 
\frac{(a_1, \ldots, a_{x_{i+r-1}})}{(a_1, \ldots, a_{i+r})} = \frac{(a_1, \ldots, a_{i+1})}{(a_1, \ldots, a_{i+r})},$$
so
\begin{equation}\label{ind1}
	\frac{(a_1,\ldots, a_{i+1}, a_{j+1})}{(a_1, \ldots, a_{i+r}, a_{j+1})} | x_{i+1,i+1} \cdots x_{i+r-1,i+r-1}.
\end{equation}
On the other hand, since $j-(i+r)<k$, our induction hypothesis gives
\begin{equation}\label{ind2}
	\frac{(a_1, \ldots, a_{i+r}, a_{j+1})}{(a_1, \ldots , a_{j+1})} | M_{i+r,j}.
\end{equation}
It follows from (\ref{ind1}) and (\ref{ind2}) that the middle terms of (\ref{expans}) are all multiples of 
$\frac{(a_1, \ldots, a_{i+1}, a_{j+1})}{(a_1, \ldots, a_{j+1})}$. We next verify that the last term of (\ref{expans}) is also a multiple of 
$\frac{(a_1, \ldots, a_{i+1}, a_{j+1})}{(a_1, \ldots, a_{j+1})}$. We have
$$x_{i+1,i+1} \cdots x_{j-1,j-1} x_{i,j} = 
\frac{(a_1, \ldots, a_{i+1})}{(a_1, \ldots, a_{i+2})} \cdots \frac{(a_1, \ldots, a_{j-1})}{(a_1, \ldots, a_j)} x_{i,j} = 
\frac{(a_1, \ldots, a_{i+1})}{(a_1, \ldots, a_{j})} x_{i,j},$$
which is divisible by $\frac{(a_1, \ldots, a_{i+1},a_{j+1})}{(a_1, \ldots, a_{j+1})}$ because 
$$\frac{(a_1, \ldots, a_{i+1},a_{j+1})(a_1, \ldots, a_j)}{(a_1, \ldots, a_{j+1})} = 
\text{lcm}\left( (a_1, \ldots, a_{i+1},a_{j+1}), (a_1, \ldots, a_j) \right)$$
is a factor of $(a_1, \ldots, a_{i+1})$, which, as above, implies that
$$\frac{(a_1, \ldots, a_{i+1},a_{j+1})}{(a_1, \ldots, a_{j+1})} | \frac{(a_1, \ldots, a_{i+1})}{(a_1, \ldots, a_j)}.$$
This completes the verification of (\ref{div1}). We next prove that
\begin{equation}\label{div2}
	\frac{(a_1, \ldots a_i, a_{j+1})}{(a_1, \ldots a_{i+1}, a_{j+1})} | M_{i,j}.
\end{equation}
Combining (\ref{div1}) and (\ref{div2}) will yield the desired result
$$\frac{(a_1, \ldots a_i, a_{j+1})}{(a_1, \ldots a_{j+1})} | M_{i,j}.$$
Since $AC=D$, comparing their $(i,j)$ entries for $1 \leq i<j < n$, gives
$$ x_{i,i} C_{i,j} + x_{i,i+1} C_{i+1,j} + \cdots + x_{i,j} C_{jj} = 0.$$
Multiplying both sides by $a_{i+1}$, and using (\ref{xii}) and (\ref{xij}), we obtain
\begin{multline}\label{prod}
	a_{i+1} x_{i,i} C_{i,j}\equiv \left( a_2 x_{1,i+1} + \cdots + a_i x_{i-1, i+1} + a_{i+2} x_{i+1,i+1}\right) C_{i+1,j}\\
	+\left( a_2 x_{1,i+2} + \cdots + a_i x_{i-1, i+2} + a_{i+2} x_{i+1,i+2}+ a_{i+3} x_{i+2,i+2}\right) C_{i+2,j}\\
	 + \cdots\\
	 +\left( a_2 x_{1,j} + \cdots + a_i x_{i-1, j} + a_{i+2} x_{i+1,j}+ \cdots + a_{j+1} x_{j,j}\right) C_{j,j}\mod a_1.
\end{multline}
For $1\leq r \leq j-i$, the term involving $a_{i+r}$ is equal to
$$a_{i+r} \left( \sum_{t=i+r-1}^{j}x_{i+r-1,t} C_{t,j} \right) = a_{i+r} D_{i+r-1,j} =0.$$
Thus (\ref{prod}) reduces to
\begin{multline}\label{prod1}
	a_{i+1} x_{i,i} C_{i,j}\equiv a_2\left( x_{1,i+1} C_{i+1,j}+ \cdots + x_{1,j} C_{j,j} \right)\\
	+a_3 \left( x_{2,i+1} C_{i+1,j}+ \cdots + x_{2,j} C_{j,j} \right)\\
	+ \cdots\\
	+a_i\left( x_{i-1,i+1} C_{i+1,j}+ \cdots + x_{i-1,j} C_{j,j} \right) +a_{j+1} x_{j,j} C_{j,j} \mod a_1.
	\end{multline}
We claim that the left hand side of (\ref{prod1}), as well as each term on the right hand
of (\ref{prod1}), including~$a_1$, is divisible by 
$\frac{a_1}{\text{lcm}\left( (a_1, \ldots, a_{i+1}),(a_1,a_{j+1}) \right)}$. Indeed, by (\ref{xii}) and (\ref{relcij}), 
the left hand side of (\ref{prod1}) is equal to
\begin{equation}\label{uno}
	\begin{split}
		a_{i+1} x_{i,i} C_{i,j} & = (-1)^{i+j}a_{i+1} \frac{a_1 (a_1, \ldots, a_{j+1})}{(a_1, \ldots, a_{i+1})(a_1,a_{j+1})} M_{i,j}\\
		& =  (-1)^{i+j}a_{i+1} \frac{a_1}{\text{lcm}\left( (a_1, \ldots, a_{i+1}),(a_1,a_{j+1}) \right)}  
		\frac{M_{i,j}}{\frac{(a_1, \ldots, a_{i+1},a_{j+1})}{(a_1, \ldots, a_{j+1})}}.
	\end{split}
\end{equation}
It follows from (\ref{div1}) that both sides of (\ref{uno})   are multiples of 
$\frac{a_1}{\text{lcm}\left( (a_1, \ldots, a_{i+1}),(a_1,a_{j+1}) \right)}$. From (\ref{relcij}), for  $1 \leq r \leq j-i-1$,  
we have
\begin{equation}\label{dos}
	\begin{split}
C_{i+r,j} & =(-1)^{i+r+j}\frac{a_1 (a_1, \ldots, a_{j+1})}{ (a_1, \ldots, a_{i+r}) (a_1, a_{j+1})} M_{i+r,j}\\
& = (-1)^{i+r+j} \frac{a_1}{\text{lcm}\left( (a_1, \ldots, a_{i+r}),(a_1,a_{j+1}) \right)}  
\frac{M_{i+r,j}}{\frac{(a_1, \ldots, a_{i+r},a_{j+1})}{(a_1, \ldots, a_{j+1})}}.
\end{split}
\end{equation}
Here $\frac{(a_1, \ldots, a_{i+r},a_{j+1})}{(a_1, \ldots, a_{j+1})}$ divides $M_{i+r,j}$ by induction hypothesis.
Since $$\frac{a_1}{\text{lcm}\left( (a_1, \ldots, a_{i+1}),(a_1,a_{j+1}) \right)} |  
\frac{a_1}{\text{lcm}\left( (a_1, \ldots, a_{i+r}),(a_1,a_{j+1}) \right)},$$ we deduce from (\ref{dos}) that
$$\frac{a_1}{\text{lcm}\left( (a_1, \ldots, a_{i+1}),(a_1,a_{j+1}) \right)} | C_{i+r,j}.$$
In addition,  $\frac{a_1}{\text{lcm}\left( (a_1, \ldots, a_{i+1}),(a_1,a_{j+1}) \right)}$ divides 
$\frac{a_1}{\text{lcm}\left( (a_1, \ldots, a_j),(a_1,a_{j+1}) \right)}= C_{j,j}$ as well as $a_1$. This proves the claim.

We next divide (\ref{prod1}) by  $\frac{a_1}{\text{lcm}\left( (a_1, \ldots, a_{i+1}),(a_1,a_{j+1}) \right)}$ and then multiply it by 
$\frac{(a_1, \ldots, a_{i+1},a_{j+1})}{(a_1, \ldots, a_{j+1})}$. 
By (\ref{uno}), the resulting left hand side will be equal to $(-1)^{i+j}a_{i+1} M_{i,j}$. Regarding the right hand side, (\ref{dos})
implies that 
$$
(-1)^{i+r+j}\frac{(a_1, \ldots, a_{i+1},a_{j+1})}{(a_1, \ldots, a_{j+1})}\frac{C_{i+r,j}}
	{\frac{a_1}{\text{lcm}\left( (a_1, \ldots, a_{i+1}),(a_1,a_{j+1}) \right)}},\quad 1 \leq r \leq j-i-1,
$$
is equal to
\begin{equation*}
	\frac{(a_1, \ldots, a_{i+1},a_{j+1})}{(a_1, \ldots, a_{j+1})}\frac{\text{lcm}
	\left( (a_1, \ldots, a_{i+1}),(a_1,a_{j+1}) \right)}{\text{lcm}\left( (a_1, \ldots, a_{i+r}),(a_1,a_{j+1}) \right)} \frac{M_{i+r,j}}
	{\frac{(a_1, \ldots, a_{i+r},a_{j+1})}{(a_1, \ldots, a_{j+1})}}=\frac{(a_1, \ldots, a_{i+1})}{(a_1, \ldots, a_{i+r})} M_{i+r, j}.
\end{equation*}
Moreover, note that
\begin{equation*}
	\begin{split}
		\frac{(a_1, \ldots, a_{i+1},a_{j+1})}{(a_1, \ldots, a_{j+1})} \frac{C_{j,j}}{\frac{a_1}{\text{lcm}\left( (a_1, \ldots, a_{i+1}),(a_1,a_{j+1}) \right)}} & = \frac{(a_1, \ldots, a_{i+1},a_{j+1})}{(a_1, \ldots, a_{j+1})} \frac{\text{lcm}\left( (a_1, \ldots, a_{i+1}),(a_1,a_{j+1}) \right)} {\text{lcm}\left( (a_1, \ldots, a_j),(a_1,a_{j+1}) \right)}\\
		& = \frac{(a_1, \ldots, a_{i+1})}{(a_1, \ldots, a_{j})},
\end{split}
\end{equation*}
and
\begin{equation*}
	\begin{split}
\frac{(a_1, \ldots, a_{i+1},a_{j+1})}{(a_1, \ldots, a_{j+1})} \frac{a_1}{\frac{a_1}{\text{lcm}\left( (a_1, \ldots, a_{i+1}),(a_1,a_{j+1}) \right)}} & = \frac{(a_1, \ldots, a_{i+1},a_{j+1})}{(a_1, \ldots, a_{j+1})} \text{lcm}\left( (a_1, \ldots, a_{i+1}),(a_1,a_{j+1}) \right)\\
& = (a_1,a_{j+1}) \frac{(a_1,\ldots,a_{i+1})}{(a_1,\ldots,a_{j+1})}.
\end{split}
\end{equation*}



Combining the above information, we find that (\ref{prod1}) reduces to
\begin{equation}
	\begin{split}
	(-1)^{i+j}a_{i+1} M_{i,j} & \equiv  a_2\left( (-1)^{i+1+j}x_{1,i+1} M_{i+1,j}+ \cdots -\frac{(a_1, \ldots, a_{i+1})}{(a_1, \ldots, a_{j-1})} x_{1,j-1} M_{j-1,j} + \frac{(a_1, \ldots, a_{i+1})}{(a_1, \ldots, a_{j})} x_{1,j}\right)\\
	& + \cdots\\
	& + a_i\left( (-1)^{i+1+j} x_{i-1,i+1} M_{i+1,j}+ \cdots - \frac{(a_1, \ldots, a_{i+1})}{(a_1, \ldots, a_{j-1})} x_{i-1,j-1} M_{j-1,j} + \frac{(a_1, \ldots, a_{i+1})}{(a_1, \ldots, a_{j})} x_{1,j}\right) \\
	& + a_{j+1} \frac{(a_1, \ldots, a_{i+1})}{(a_1, \ldots, a_{j+1})} 
	\mod (a_1,a_{j+1})\frac{(a_1, \ldots, a_{i+1})}{(a_1, \ldots, a_{j+1})}
\end{split}	
\end{equation}
The presence of $a_2,\dots,a_{j+1}$ ensures that each term in the right hand side 
is divisible by $(a_1, \ldots a_i, a_{j+1})$. Moreover, it is clear that $(a_1, \ldots a_i, a_{j+1})$ is a factor of
$(a_1,a_{j+1})$ and hence of $(a_1,a_{j+1})\frac{(a_1, \ldots, a_{i+1})}{(a_1, \ldots, a_{j+1})}$. Therefore
 $$(a_1, \ldots a_i, a_{j+1}) | a_{i+1} M_{i,j}$$
 and hence
 $$\frac{(a_1, \ldots a_i, a_{j+1})}{(a_1, \ldots a_{i+1}, a_{j+1})} | \frac{a_{i+1}}{(a_1, \ldots a_{i+1}, a_{j+1})} M_{i,j}.$$
But  $\frac{(a_1, \ldots a_i, a_{j+1})}{(a_1, \ldots a_{i+1}, a_{j+1})}$ and
$\frac{a_{i+1}}{(a_1, \ldots a_{i+1}, a_{j+1})}$ are relatively prime, so (\ref{div2}) holds.
\end{proof}

We return to the problem of describing the structure of the finitely generated torsion $R$-module $Q=S/U_1$.
This is completely determined by the $p$-elementary divisors of $Q$ for each prime $p\in R$, so we
are reduced to finding these. We fix for this purpose a prime $p\in R$, and write $R_p$ and $Q_p$ for the localizations
of $R$ and $Q$ at $p$, so that $Q_p$ is a module over the local principal ideal domain $R_p$,
and $R$ can be viewed as a subring of $R_p$. For any $R$-module $Z$, we write $Z[p]$ for the $p$-component of $Z$, i.e. 
$Z[p]=\{z\in Z\,|\, p^s z=0\text{ for some }s\geq 1\}$. We readily see that the map $Z[p]\to Z_p[p]$, given by $x\to \frac{x}{1}$,
is an isomorphism of $R$-modules. This applies, in particular, to $Z=Q$. 
Now $Q_p[p]$ is the direct sum of $R_p$-modules of the form $R_p/R_p p^s$, and it
is easy to see that $R_p/R_p p^s \cong R/R p^s$ as $R$-modules. It follows that the $p$-elementary divisors of $Q_p$ as a 
$R_p$-module
coincide with the $p$-elementary divisors of $Q$ as a $R$-module. We may thus restrict our attention to the former.

Since $(S/U_1)_p\cong S_p/(U_1)_p$ as $R_p$-modules, we may now apply the construction from the beginning of this section, with $R$ replaced by $R_p$, $S$ by $S_p$, and $U_1$ by $(U_1)_p$. The elements $a_2,\dots ,a_n$ belong to $R_p$ and are linearly ordered there by divisibility. The isomorphism type of $S/U_1$ does not depend on the order of $a_2,\dots ,a_n$. For the purpose of computing the elementary divisors of $(S/U_1)_p$, we may therefore assume
without loss that $a_2|\cdots|a_n$ in $R_p$.

Once again, we take $A$ to be as in (\ref{matriz}) when $M=N$, and construct the upper triangular matrix $C$ as before, now with entries in $R_p$. From (\ref{dis}) and the divisibility relations among the $a_i$, the diagonal entries of $C$ are given by
$$
d_1=1,\ d_i=\frac{a_1(a_1,\dots ,a_{i+1})}{(a_1,\dots ,a_{i})(a_1,a_{i+1})} = 
\frac{a_1(a_1,a_2)}{(a_1,a_2)(a_1,a_{i+1})} = \frac{a_1}{(a_1,a_{i+1})},\quad 2\leq i<n.
$$
Since $a_{i+1}|a_{i+2}$, we have $(a_1,a_{i+1})|(a_1,a_{i+2})$, and therefore $d_{i+1}|d_i$. Thus
\begin{equation}
\label{ddiven}
d_{n-1}|\cdots|d_2.
\end{equation}

Making use of (\ref{ddiven}) and (\ref{Cij}), we can eliminate all entries $C_{i,j}$, with $i<j$, through a sequence of elementary row and column operations on $C$. Since $d_1=1$, we first clear the off-diagonal entries from Row 1 by adding multiplies of Column 1 to subsequent columns. Within each column of the resulting upper triangular matrix, each of the off-diagonal entries is a $R_p$-multiple of the entry in the diagonal position. We clear these columns one at a time, from Column $n-1$ to Column 3, to obtain the diagonal matrix with diagonal entries $1,d_2,\dots ,d_{n-1}$. We conclude that when $a_2,\dots ,a_n$ are ordered according to divisibility in $R_p$, the $p$-parts of the corresponding elements $d_2,\dots ,d_{n-1}$ of $R$ are the $p$-elementary divisors of $(S/U_1)_p$.

For any unimodular vector $(a_1,a_2,\dots ,a_n)\in R^n$, we wish to use the elementary divisors of $S/U_1$ to describe the structure of this module in terms of the elements $d_2,\dots ,d_{n-1}$ of $R$, defined as in (\ref{dis}). The remaining obstacle to this goal is that for each prime $p$, our description of the $p$-elementary divisors depends on a choice of ordering of $a_2,\dots ,a_n$ that is particular to $p$. While the structure of $S/U_1$ does not depend on the order of $a_2,\dots ,a_n$, the elements $d_i$ do. For example, $(d_2,d_3)=(2,3)$ if $R=\Z,\ n=4$ and $(a_1,a_2,a_3,a_4)=(12,15,10,20)$, but $(d_2,d_3)=(1,6)$ if $(a_1,a_2,a_3,a_4)=(12,20,10,15)$. We now show that the $p$-parts of $d_2,\dots ,d_{n-1}$ coincide for all orderings of $a_2,\dots ,a_n$.

  \begin{lemma}\label{permutation}
Let $\pi$ be any permutation of $\{2,\dots ,n\}$. Let $d_2,\dots ,d_{n-1}$ and $d_2',\dots ,d_{n-1}'$ correspond as in (\ref{dis}) to $(a_2,\dots ,a_n)$ and $(a_{\pi(2)},\dots ,a_{\pi(n)})$ respectively. Then for any prime $p$ in $R$, the lists of $p$-parts of $d_2,\dots ,d_{n-1}$ and of $d_2',\dots ,d_{n-1}'$ are permutations of each other. 
  \end{lemma}
  \begin{proof}
    Since every permutation of $\{2,\dots ,n\}$ is a composition of transpositions of consecutive integers, it is sufficient to prove the lemma for a permutation of this type. We assume that $\pi$ is the transposition $(t \ \ t+1)$, where $2\le t\le n-1$, and we fix a prime $p$ that divides $a_1$ in $R$. We write $b_i$ for the $p$-part of $a_i$, and write $q_i$ and $q'_i$ respectively for the $p$-parts of $d_i$ and $d'_i$. We write $b$ for the element $(b_1,\dots ,b_{t-1})$ of $R$. 
    It is immediate from (\ref{dis}) that $q_i=q'_i$ for all $i\not\in\{t-1,t\}$, and that $q_2=q'_2$ if $t=2$. For $t\ge 3$, we have
  $$
  q_{t-1}=\frac{b_1(b,b_t)}{b(b_1,b_t)},\ q_t=\frac{b_1(b,b_t,b_{t+1})}{(b,b_t)(b_1,b_{t+1})}, \
  q'_{t-1}=\frac{b_1(b,b_{t+1})}{b(b_1,b_{t+1})},\ q'_t=\frac{b_1(b,b_t,b_{t+1})}{(b,b_{t+1})(b_1,b_{t})}.
  $$
These expressions simplify as follows, according to order of the elements $b,b_t,b_{t+1}$ under divisibility. \\
\emph{Case 1}: If $b|b_t$ then
$$
q_{t-1}=q'_t=\frac{b}{(b_1,b_t)},\ \mathrm{and}\ q_t=q'_{t-1}=\frac{b_1(b,b_{t+1})}{b(b_1,b_{t+1})}.
  $$
  \emph{Case 2}: If $b\not | b_t$ and $b|b_{t+1}$, then $b_t|b$ and 
 $$
  q_{t-1}=q'_t=\frac{b_1}{b},\ \mathrm{and}\ q_t=q'_{t-1}=\frac{b_1}{(b_1,b_{t+1})}.
$$
  \emph{Case 3}: If $b\not | b_t$ and $b\not |b_{t+1}$, then $b_t|b$, $b_{t+1}|b$ and 
 $$
  q_{t-1}=q'_{t-1}=\frac{b_1}{b},\ \mathrm{and}\ q_t=q'_{t}=\frac{b_1(b_t,b_{t+1})}{b_tb_{t+1}}.
$$
Thus the list of $q'_i$ is either identical to the list of $q_i$, or differs from it by the transposition of $q_{t-1}$ and $q_t$.
\end{proof}

We are now in a position to state our main theorem on the structure of $S/U_1$, under the general hypothesis of a unimodular vector $(a_1,a_2,\dots ,a_n)$ in $R^n$.

\begin{theorem}\label{S/U_1} Suppose that $R$ is a principal ideal domain, $a_1\neq 0$, $n>2$, and $p\in R$
is a prime. Then the $p$-elementary divisors of
$S/U_1$ are given by the $p$-parts of $d_2,\dots,d_{n-1}$, as defined in (\ref{dis}). Moreover, we have
$$
S/U_1\cong R/Rd_2\oplus\cdots\oplus Rd_{n-1}.
$$
\end{theorem}

\begin{proof} The first assertion follows from Lemma \ref{permutation} and the argument preceding it, while the second follows from the first
by means of the prime factorization of each $d_i$.
\end{proof}

\section{An example}

Consider the case $R =\Z$, $n=4$, and $(a_1,a_2,a_3,a_4)=(12,4,2,3)$. The corresponding homogeneous linear diophantine equation is
\begin{equation}
\label{eq}
12 X_1 + 4 X_2 + 2 X_3 + 3 X_4 =0.
\end{equation}
The following vectors form a $\Q$-basis for the space $T$ of solutions of (\ref{eq}) in $\Q^4$:
\begin{equation}
\label{defw}
w(1,2) = \left( -\frac{1}{3},1,0,0\right) , \; w(1,3) = \left( -\frac{1}{6},0,1,0\right) , \; 
w(1,4)= \left( -\frac{1}{4}, 0,0,1\right).
\end{equation}
Let $S=T\cap\Z^4$, namely the $\Z$-module of solutions of (\ref{eq}) in $\Z^4$, and let 
$W$ be the $\Z$-span of $w(1,2),w(1,3),w(1,4)$. Let us use Theorem \ref{basis} to obtain a $\Z$-basis of $S$ from these vectors.

According to Theorem \ref{ws}, we have
\begin{equation}
\label{wmods}
W/S\cong \Z/12\Z.
\end{equation}
Since $\{w(1,2),w(1,3),w(1,4)\}$ is a $\Q$-basis of $T$, (\ref{defw}) says that 
$\alpha_2 w(1,2)+ \alpha_3 w(1,3) + \alpha_4 w(1,4) \in S$
if and only if $\alpha_2,\alpha_3,\alpha_4 \in \Z$ and
\begin{equation}
\label{alfa}
4\alpha_2+ 2\alpha_3 + 3\alpha_4\equiv 0\mod 12.
\end{equation}
On the other hand, we have
\begin{equation}
\label{dei}
(a_1,a_2) = 4, (a_1,a_3) = 2, (a_1,a_4) = 3, (a_1,a_2,a_3)=2, (a_1,a_2,a_3,a_4)=1.
\end{equation}
In particular,
$$
a_1/(a_1,a_2)=3, (a_1,a_2)/(a_1,a_2,a_3)=2, (a_1,a_2,a_3)/(a_1,a_2,a_3,a_4)=2.
$$
We look for $a,b,c\in\Z$ such that the columns of the matrix
$$A = \left( \begin{array}{ccc}
	3 & a & b\\
	0 & 2 & c\\
	0 & 0 & 2
\end{array}\right)
$$
are the coordinates of basis vectors $z_2,z_3,z_4$ of $S$ relative to $w(1,2),w(1,3),w(1,4)$.
Since $|A|=12$, it follows from (\ref{wmods}) and (\ref{alfa}) that all we need to do is to make sure that
$$
4a+4\equiv 0\mod 12\text{  and }4b+2c+6\equiv 0\mod 12.
$$
It is clear that $(a,b,c)=(-1,1,1)$ satisfies these requirements, so
\begin{equation}
\label{dez}
z_2=3w(1,2),\; z_3=-w(1,2)+2w(1,3),\; z_4=w(1,2)+w(1,3)+2w(1,4)
\end{equation}
form a $\Z$-basis of $S$. On the other hand, by definition, we have
\begin{equation}
\label{dev}
v(1,2)=12w(1,2),\; v(1,3)=12w(1,3),\; v(1,4)=12w(1,4),
\end{equation}
and Theorem \ref{S/S_1} ensures that
\begin{equation}
\label{smods1}
S/S_1\cong \Z/12\Z\oplus \Z/\Z 12.
\end{equation}
We can easily confirm this fact. Indeed, it follows from (\ref{dez}) and (\ref{dev}) that the matrix whose
columns are the the coordinates of $v(1,2),v(1,3),v(1,4)$ relative to $z_2,z_3,z_4$ is
\begin{equation}
\label{deb}
\left( \begin{array}{ccc}
	4 & 2 & -3\\
	0 & 6 & -3\\
	0 & 0 & 6
\end{array}\right).
\end{equation}
Since the Smith Normal Form of this matrix is $\mathrm{diag}(1,12,12)$, it follows that (\ref{smods1}) is correct.

In view of (\ref{dis}) and (\ref{dei}), Theorem \ref{S/U_1} predicts that
\begin{equation}
\label{smodu1}
S/U_1\cong \Z/3\Z\oplus \Z/2\Z.
\end{equation}
Let us confirm this. Indeed, the definition of $u(1,2),u(1,3),u(1,4)$ as well as (\ref{dez}) and (\ref{deb}) yield
\begin{equation}
\label{deu}
u(1,2)=3w(1,2)=z_2,\; u(1,3)=6w(1,3)=3z_3+z_2,\; u(1,4)=4w(1,4)=2z_4-z_3-z_2.
\end{equation}
Thus, the matrix whose columns are the the coordinates of $u(1,2),u(1,3),u(1,4)$ relative to $z_2,z_3,z_4$ is
$$
\left( \begin{array}{ccc}
	1 & 1 & -1\\
	0 & 3 & -1\\
	0 & 0 & 2
\end{array}\right).
$$
This is clearly equivalent to $\mathrm{diag}(1,3,2)$, thereby confirming (\ref{smodu1}).

Since $(a_3,a_4)=1$, it follows from Lemma \ref{gen} that the following vectors $\Z$-span $S$:
$$
v(1,3)=(-2,0,12,0),\; v(1,4)=(-3,0,0,12),\; v(2,3)=(0,-2,4,0),
$$
$$
v(2,4)=(0,-3,0,4),\; v(3,4)=(0,0,-3,2).
$$
Theorems \ref{rela} and \ref{rela2} predict the following defining relations among these vectors:
\begin{equation}
\label{pela}
3v(1,3)-2v(1,4)+12v(3,4)=0,\; 3v(2,3)-2v(2,4)+4v(3,4)=0.
\end{equation}
These are clearly valid relations. Moreover, in the notation used in the proof of Theorem \ref{rela2}, we have
$(n,m,d,e,r)=(4,2,5,2,2)$. Moreover, the matrix whose columns are the coordinates of the $y(i,j,k)$ relative to the $x(i,j)$ is
$$
\left( \begin{array}{cc}
	3 & 0\\
	-2 & 0\\
	0 & 3\\	
	0 & -2\\
	12 & 4
\end{array}\right),
$$
whose Smith Normal Form is $\mathrm{diag}(1,1)$, as required in the proof of Theorem \ref{rela2}  to confirm that
(\ref{pela}) are indeed defining relations.

\medskip

\noindent{\bf Acknowledgment.} We thank the referee for a careful reading of the paper and valuable suggestions for changes.


\end{document}